\documentclass[aop,preprint]{imsart}

\RequirePackage[OT1]{fontenc}
\RequirePackage{amsthm,amsmath}
\RequirePackage[numbers]{natbib}
\RequirePackage[colorlinks,citecolor=blue,urlcolor=blue]{hyperref}

\usepackage[left=1in,right=1in,bottom=1in,top=1in]{geometry}
\usepackage{amsmath}
\usepackage{amsfonts}
\usepackage{amssymb}
\usepackage{amsthm}
\usepackage{young}
\usepackage[vcentermath]{youngtab}
\usepackage{tikz}

\usepackage{color}
\usepackage[leqno]{amsmath}
\usepackage{shuffle}
\usepackage{tikz}
\usetikzlibrary{arrows}
\usepackage{enumitem}

\makeatletter
\newtheorem*{rep@theorem}{\rep@title}
\newcommand{\newreptheorem}[2]{%
\newenvironment{rep#1}[1]{%
 \def\rep@title{#2 \ref{##1}}%
 \begin{rep@theorem}}%
 {\end{rep@theorem}}}
\makeatother

\usepackage{graphicx}

\newcommand{\diag}{\mbox{diag}}

\newtheorem{theorem}{Theorem}
\newreptheorem{theorem}{Theorem}
\newtheorem{lemma}[theorem]{Lemma}
\newtheorem{example}[theorem]{Example}

\newtheorem{proposition}[theorem]{Proposition}
\newtheorem{corollary}[theorem]{Corollary}

\DeclareMathOperator{\Diag}{diag}
\DeclareMathOperator{\eig}{eig}

\DeclareMathOperator{\TV}{T.V.}

\def\unprotectedboldentry#1{\textcolor{Red}{\large{#1}}}
\def\boldentry{\protect\unprotectedboldentry}
\usepackage{tikz}
\usetikzlibrary{calc}
\usepackage{ifthen}
\newcommand{\tikztableauinternal}[1]{
    \def\newtableau{#1}
    \coordinate (x) at (-0.5,0.5);
    \coordinate (y) at (-0.5,0.5);
    \foreach \row in \newtableau {
        \coordinate (x) at ($(x)-(0,1)$);
        \coordinate (y) at (x);
        \foreach \entry in \row {
            \ifthenelse{\equal{\entry}{X}}
               {
                \node (y) at ($(y) + (1,0)$) {};
                \fill[color=gray!10] ($(y)-(0.5,0.5)$) rectangle +(1,1);
                \draw[color=gray, dotted] ($(y)-(0.5,0.5)$) rectangle +(1,1);
               }
               {
                \ifthenelse{\equal{\entry}{\boldentry X}}
                   {
                    \node (y) at ($(y) + (1,0)$) {};
                    \fill[color=gray] ($(y)-(0.5,0.5)$) rectangle +(1,1);
                    \draw ($(y)-(0.5,0.5)$) rectangle +(1,1);
                   }
                   {
                    \node (y) at ($(y) + (1,0)$) {\entry};
                    \draw ($(y)-(0.5,0.5)$) rectangle +(1,1);
                   }
               }
            }
        }
}

\newcommand{\tikztableau}[2][scale=0.6,every node/.style={font=\small}]{
    \begin{array}{c}
    \begin{tikzpicture}[#1]
        \tikztableauinternal{#2}
    \end{tikzpicture}
    \end{array}}
\newcommand{\tikztableausmall}[1]{\tikztableau[scale=0.45,every node/.style={font=\rm\small}]{#1}}

\newlength\cellsize 
\savebox2{%
\begin{picture}(12,12)
\put(0,0){\line(1,0){12}}
\put(0,0){\line(0,1){12}}
\put(12,0){\line(0,1){12}}
\put(0,12){\line(1,0){12}}
\end{picture}}

\newcommand\cellify[1]{\def\thearg{#1}\def\nothing{}%
\ifx\thearg\nothing
\vrule width0pt height\cellsize depth0pt\else
\hbox to 0pt{\usebox2\hss}\fi%
\vbox to 12\unitlength{
\vss
\hbox to 12\unitlength{\hss$#1$\hss}
\vss}}

\newcommand\tableau[1]{\vtop{\let\\=\cr
\setlength\baselineskip{-12000pt}
\setlength\lineskiplimit{12000pt}
\setlength\lineskip{0pt}
\halign{&\cellify{##}\cr#1\crcr}}}

\startlocaldefs
\endlocaldefs

\begin{document}

\begin{frontmatter}

\title{Cutoff for random to random card shuffle}
\runtitle{Cutoff for random to random card shuffle}


\begin{aug}
\author{\fnms{Megan} \snm{Bernstein}\thanksref{t1}\ead[label=e1]{bernstein@math.gatech.edu}}
\and
\author{\fnms{Evita} \snm{Nestoridi}
\ead[label=e3]{exn@princeton.edu}
}

\thankstext{t1}{The first author is partially supported by NSF Grant $DMS1344199$}

\affiliation{Georgia Institute of Technology and Princeton University}

\address{Email:bernstein@math.gatech.edu}

\address{Email:exn@princeton.edu}
\end{aug}

\begin{abstract}In this paper, we use the eigenvalues of the random to random card shuffle to prove a sharp upper bound for the total variation mixing time. Combined with the lower bound due to Subag, we prove that this walk exhibits cutoff at $\frac{3}{4} n \log n - \frac{1}{4}n\log\log{n}$ with window of order $n$, answering a conjecture of Diaconis.
\end{abstract}



\end{frontmatter}

\section{Introduction}\label{intro}
One step of the random to  random card shuffle consists of picking a card and a position of the deck uniformly and independently at random and moving that card to that position. It was introduced by Diaconis and Saloff-Coste \cite{DS}, who proved that the mixing time is $O(n \log n)$. Uyemura-Reyes \cite{JC} proved that the mixing time is between $\frac{1}{2} n \log n$ and $4n \log n$. Saloff-Coste and Z{\'u}{\~n}iga \cite{LZ} improved the upper bound constant by showing that the mixing time is bounded by above by $2 n \log n$. Subag \cite{Lower} proved a lower bound for the mixing time of the form $\frac{3}{4} n \log n - \frac{1}{4}n\log{n} -cn$. The best previously known bound for the total variation distance mixing time was $1.5324n \log n$ and is due to Morris and Qin \cite{MorrisQin}. 

In this paper we prove that the random to random shuffle exhibits cutoff at $\frac{3}{4}n\log n- \frac{1}{4}n\log\log{n} $ with window of order $n$, which answers a conjecture of Diaconis \cite{DiaConj}. Diaconis's conjecture of cutoff after $(1\pm o(1))\frac{3}{4}n\log{n}$ steps was based on a partial diagonalization of the transition matrix by Uyemura-Reyes in his thesis \cite{JC}. 

Recently, Dieker and Saliola \cite{DSaliola} used representation theory arguments to find the eigenvalues and eigenvectors of this card shuffle. We use this information to give a tight bound for the total variation distance, defined as
$$||P_{x}^{*t}-\pi ||_{\TV}= \frac{1}{2} \sum_{g \in S_n} |P^{*t}_{x}(g)- \pi(g)|,$$
where $P_{x}^{*t}(g)$ is the probability of moving from $x\in S_n$ to $g \in S_n$ after $t$ steps and $\pi$ is the uniform measure on $S_n$.
Our first result is:
 \begin{theorem}\label{rtr}
For the random to random card shuffle, if $t=  \frac{3}{4}n \log n - \frac{1}{4}n\log\log{n}+c n$ with $c \geq 2$ and $n$ is sufficiently large, then
\begin{equation*}
||P_{x}^{*t}-\pi ||_{\TV} \leq  e^{-c},
\end{equation*}
for every $x\in S_n$.
\end{theorem} 
Theorem \ref{rtr}, in combination with Subag's lower bound \cite{Lower}, proves the existence of cutoff for the random to random card shuffle. Theorem \ref{rtr} is proven by bounding the eigenvalues of the random to random card shuffle. There has been much interest in the cutoff phenomenon, which is believed to be a common behavior of Markov chains. A heuristic given in \cite{DiaconisSurvey} suggests reversible card shuffles such as random to random should mix with cutoff; however, only a small number of Markov chains have been shown to mix with cutoff. 

We also consider the following generalization. Let $\nu=[\nu_1,\ldots, \nu_m]$ be a partition of $n$. Then we can perform the random to random card shuffle on a deck which has $\nu_i$ cards of type $i$, where $i=1,\ldots ,m$. We then say that the deck has evaluation $\nu$. 
  \begin{theorem}\label{word}
For the random to random shuffle on a deck with evaluation $\nu=[\nu_1,\ldots, \nu_m]$ with $\nu_1 \neq n$, we have that 
\begin{enumerate}
\item For $t=  \frac{3}{4} n\log n - \frac{1}{4}n\log\log{n} +cn$, then 
\begin{equation*}
||P_{x}^{*t}-\pi ||_{2}   \leq 2e^{-c},
\end{equation*}
for every $x\in S_n$ and every $c\geq 2$.
\item For $t= \frac{n}{4} \log m + \frac{n}{8} \log n -cn$, then 
\begin{equation*}
\left|\left| \frac{P_{x}^{*t}}{\pi} -1 \right|\right|_2  \geq \frac{1}{2}e^{c},
\end{equation*}
where $x\in S_n$, $c>0$ and the $l^2$ norm is with respect to the stationary distribution.
\end{enumerate}
\end{theorem}
Since $m \leq n$, Theorem \ref{word} says that order $n \log n$ steps are necessary and sufficient for the $l^2$ distance to be small. Moreover, Theorem \ref{word} and Theorem 1.3 of Chen and Saloff-Coste \cite{CS} give that the random to random shuffle on a deck with evaluation $\nu$ will exhibit cutoff in $l^2$ distance at $Cn\log{n}$ steps for some constant $C$ that depends on $\nu$. 

The bounds of Theorem \ref{word} are not optimal for the total variation distance in the case where $m$ is a constant. For example, if  $\nu=[n-1,1]$, then we have $n-1$ cards that have label $a$ and one card labeled $b$. To get a random configuration of such a deck, we only have to wait until we touch card $b$. More precisely, if $T$ is the first time that we touch card $b$, then $T$ is a coupling time and a strong stationary time. This argument proves that the mixing time for the total variation distance and the separation distance is at most of order $n$, while Theorem \ref{word} says that $l^2$ mixing time is at least $\frac{n}{8} \log n$. 

Although proving Theorems \ref{rtr} and \ref{word} falls into the tradition of the pioneering work of spectral analysis of random walks on groups in the work of Diaconis and Shahshahani \cite{DiaSha} in bounding the mixing time of the random transposition card shuffle, it differs significantly in some respects. The analysis of both shuffles use the representation theory of the symmetric group to find a tight upper bound. The transposition walk is a conjugacy class walk, in which all elements of a conjugacy class have the same probability of being a generator of the random walk. Because of this, the walk can be diagonalized using the characters of the irreducible representations, with each irreducible representation contributing a single eigenvalue. 

The random to random shuffle is not a conjugacy class walk. For example, one conjugacy class of the symmetric group is all of the transpositions. One step of the random to random shuffle can change the order of the deck by an adjacent transposition, but never any other transposition. Since it is not a conjugacy class walk, the shuffle does not act as a constant operator on each irreducible representation. This is reflected in the eigenvalues being indexed by a pair of partitions $(\lambda, \mu)$ where $\lambda$ is indexing the irreducible representation and $\mu$ a further level of decomposition. This greatly increases the difficulty in diagonalizing and analyzing such walks.

As mentioned earlier, we make use of Theorem $5$ of \cite{DSaliola}, where Dieker and Saliola prove a formula for the eigenvalues of the random to random shuffle. The main challenge in proving Theorem \ref{rtr} is to understand how these eigenvalues and their multiplicities behave and appropriately bound them. In most card shuffling examples, such as random transpositions, the largest non-trivial eigenvalue and its multiplicity determine the mixing time. This is not the case for the random to random shuffle, as the spectrum corresponding to each irreducible representation splits into a cascade of eigenvalues, that start tightly bunched before trailing off. This feature of the spectrum, as first noticed by Uyemura-Reyes in the partial diagonalization, gives rise to the odd coefficient of $\frac{3}{4}$.  Uyemura-Reyes explicitly found $n-1$ of the non-trivial eigenvalues of the random to random shuffle each with multiplicity $n-1$, and in a computation similar to the lower bound from Theorem \ref{word}, achieved an $l^2$ lower bound and a suggestion of where an $l^2$ upper bound occurs. Unusually, for this walk, the largest non-trivial eigenvalue and its multiplicity do not determine the mixing time, as their contribution is small after $t = \frac{1}{2}n\log{n} +cn$ steps. Since the gaps between the next $\sqrt{n}$ largest eigenvalues are of smaller order than the spectral gap, their combined multiplicity of order $n^{3/2}$ takes about $\frac{3}{4}n\log{n}$ steps to control. Taking full advantage of the smaller order decay between these eigenvalues saves the additional $\frac{1}{4}n\log\log{n}$ factor. See Proposition \ref{largesteig} for full details.   

This paper is organized as following: Section \ref{prel} explains how the eigenvalues of the transition matrix $P$ can be used to bound the $l^2$ distance. In Sections \ref{eigenvalues} and \ref{boundmult}  we quickly review the work of Dieker and Saliola \cite{DSaliola} and we provide bounds for the eigenvalues and their multiplicities. In Section \ref{prf} we present the proof of Theorem \ref{rtr}. In Section \ref{generalization}, we briefly examine the case where the $n$ cards of the deck are not necessarily distinct and prove Theorem \ref{word}. On section \ref{lastsection}, we present an application of our result for the random to random card shuffle: Theorem \ref{rtr} is used to prove that the time inhomogeneous card shuffle known as card-cyclic to random mixes in at most $\frac{3}{2} n \log n $ steps.

\section{The eigenvalues and their multiplicities}\label{prel}

The random to random card shuffle is a symmetric, transitive, reversible Markov chain, so its transition matrix diagonalizes with an orthonormal eigenbasis. We now recall Lemma $12.16$ of \cite{LPW} which will allow us to use a $l^2$ bound.

\begin{lemma}\label{l2}
Let $P$ be a reversible, transitive transition matrix of a random walk on a finite state space $\Omega$ with eigenvalues $1 =\lambda_1 >\lambda_2 \geq \ldots \geq \lambda_{|\Omega|} \geq -1$, then:
\[4||P^t(x, \cdot) - \pi||^2_{\TV} \leq \left|\left| \frac{P^t(x,\cdot)}{\pi(\cdot)} -1 \right|\right|^2_2 = \sum_{j=2}^{|\Omega|} \lambda_j^{2t}\]
for every starting point $x \in \Omega$.
\end{lemma}
\subsection{The eigenvalues}\label{eigenvalues}
In this section, we will present the formula for the eigenvalues that Dieker and Saliola \cite{DSaliola} introduced. We will present some basic tools from the representation theory of the symmetric group, and we will prove a few bounds for the eigenvalues and their multiplicities that will be used to prove Theorem \ref{rtr}. 

A partition $\lambda= [\lambda_1,\ldots,\lambda_r]$ of $n$ is a list of integers satisfying $\lambda_1 \geq \cdots \geq \lambda_r >0$ and $\lambda_1 + \cdots \lambda_r = n$. We denote $|\lambda| = n$. The Young diagram of $\lambda$ consists of $r$ left justified rows of cells with $\lambda_i$ cells in row $i$. For a partition $\mu$ of $m$, we say $\mu$ is contained in $\lambda$ if $\mu_i \leq \lambda_i$ for each $i$, and the pair of partitions is called a skew shape. The skew Young diagram $\lambda/\mu$ consists of the cells of $\lambda$ not in $\mu$. Let $\lambda'$ and $\mu'$ denote the transpose partitions of $\lambda$ and $\mu$ respectively.

The skew Young diagram $\lambda/\mu$ is a horizontal strip if for each $i$, $\lambda_i' -1 \leq \mu_i'$, meaning the skew diagram has at most one box per column.
\begin{example}
If $n=5$, for
$\lambda=
\young(~~~,~~)
$
and $\mu=
\young(~~,~)
$
we have that $\lambda/ \mu= \tikztableausmall{{    X,    X,\null},{   X,\null}}$

is a skew diagram and a horizontal strip.
\end{example}
The eigenvalues of the random to random card shuffle are indexed by such pairs $(\lambda, \mu)$ where  $\lambda/\mu$ is a horizontal strip. For any Young diagram, let $(i,j)$ denote the cell that is in the $i^{th}$ row and in the $j^{th}$ column. Define the diagonal index of $\lambda$:
\begin{equation}\label{diag}
diag( \lambda )= \sum_{(i,j)\in \lambda} (j-i).
\end{equation}

\begin{example}
For
$\lambda=
\young(~~~,~~)
$
we have that $\Diag(\lambda)= 0+1+2-1+0=2$.
\end{example}
Dieker and Saliola \cite{DSaliola} found that the eigenvalue that corresponds to the pair $(\lambda, \mu)$ is:
\begin{equation}\label{eigen}
\eig( \lambda/ \mu) = \frac{1}{n^2} \left( {n+1 \choose 2} - {|\mu| + 1 \choose 2} + \Diag(\lambda) - \Diag(\mu) \right).
\end{equation}

To discuss what the multiplicity of each eigenvalue is, we need to talk about {\it desarrangement} tableaux. Firstly, a {\it tableau} of shape $\lambda$ is a filling of a Young diagram so that each cell is assigned a value. A {\it standard tableau} of shape $\lambda$ with $|\lambda| = n$ is one in which the values are from $\{1,\ldots,n \}$ and the entries of each row and column are in strictly increasing order. This necessitates that each entry is used once.
An entry $i$ of a tableau is an {\it ascent} if $i=n$ or if $i<n$ and $i+1$ occurs weakly to the north and east of $i$.  A standard tableau is a {\it desarrangement} tableau if its smallest ascent is even. This is equivalent to either there being no $(1,2)$ cell of $\lambda$ with $|\lambda|$ even or that the value of the $(1,2)$ cell is odd.
\begin{example}\label{desEx}
The tableau $
\young(13,2)
$
is a {\it desarrangement} tableau while $
\young(12,3)
$ is not. The tableaux of partitions with one row $\young(1)$ and $\young(12)$ are not {\it desarrangement} tableaux while $\young(1,2)$ is. Also, the tableau  $
\young(1,2,3)
$ is not a {\it desarrangement} tableau.
\end{example}

\begin{theorem}[Dieker-Saliola, \cite{DSaliola}]\label{DieSal}
The eigenvalues of the random to random card shuffle are indexed by $(\lambda, \mu)$ pairs where $\lambda/\mu$ is a horizontal strip. The corresponding eigenvalue is $\eig(\lambda/\mu)$  and occurs with multiplicity $d_\lambda d^\mu$ where $d_\lambda$ is the number of standard Young tableaux of shape $\lambda$ and $d^\mu$ is the number of {\it desarrangement} tableaux of shape $\mu$. 
\end{theorem}

Since there are $n!$ eigenvalues, we can only feasibly control them through upper bounds based on a few features of $\lambda$ and $\mu$. Despite the added complication of two indexing parameters, there is still a partial monotonicity relation on the eigenvalues for fixed $\lambda$ as $\mu$ changes that is simple to describe.

The following lemma says that all the eigenvalues of the random to random card shuffle are non-negative.
\begin{lemma}\label{positivity}
For all horizontal strips $\lambda/\mu$,
\begin{equation*}
\eig(\lambda/\mu) \geq 0.
\end{equation*}
\end{lemma}
\begin{proof}
This quickly follows from the random to random shuffle being the symmetrization of the random to top shuffle. Denoting the transition matrix of the random to top shuffle as $A$, $P = A^{*}A$ and therefore is positive semi-definite.
\end{proof}

\begin{proposition}\label{monotonicity}
Given Young diagrams $\lambda, \mu, \tilde{\mu}$, where $\lambda / \mu$, $\lambda / \tilde{\mu}$ are horizontal strips and $\tilde{\mu} \subseteq \mu$ with $\subseteq$ referring to containment, we have that
\[ \eig(\lambda/\mu) \leq \eig(\lambda/\tilde{\mu}).\]

\end{proposition}

\begin{proof}
Under the given assumptions it is easy to see that $\lambda/\tilde{\mu}=\lambda/\mu \oplus \mu/\tilde{\mu} $ where $\oplus$ indicates concatenation of the cells of the two skew diagrams. If $\lambda/\tilde{\mu}$ has at most one cell per column then this has to be true for $\lambda/\mu$ and $\mu/\tilde{\mu}$. Therefore, $\mu/\tilde{\mu}$ is a horizontal strip. Let $|\lambda| = n$ and $|\mu| = m$. Then, we can write
\[ n^2\eig(\lambda/\mu) = n^2\eig(\lambda/\tilde{\mu}) - m^2\eig(\mu/\tilde{\mu}),\]
which is a clear outcome of \eqref{eigen}.
Lemma \ref{positivity} says that $ \eig(\mu/ \tilde{\mu} )\geq 0,$ which finishes the proof.
\end{proof}

The following proposition studies the sizes of the eigenvalues in terms of the length of the first row of $\lambda$ and the size of $\mu$. It provides a single bound that will be sufficient to get the $\frac{3}{4}n \log n - \frac{1}{4}\log\log n + cn$ bound for the mixing time: 
\begin{proposition}\label{ieb}
For any $\lambda$ with $n- \lambda_1 = l$ and $\mu$ a partition of size $|\mu| = l + k$ for which $\lambda/\mu$ is a horizontal strip,
\[ \eig({\lambda/\mu}) \leq 1 - \frac{l}{n} - \frac{k^2 + kl}{n^2} .\]
\end{proposition}
\begin{proof}

We have that
\[ n^2\eig(\lambda,\mu) = {n + 1 \choose 2} - {k+l + 1 \choose 2} + \diag(\lambda/\mu),\]
with $\diag(\lambda/\mu) = \sum_{(i,j) \in \lambda/\mu} (j-i)$ where box $(i,j)$ is in row $i$ and column $j$ of the shape. When $|\mu| < l$, $\lambda/\mu$ is never a horizontal strip and so does not index an eigenvalue. For each $\lambda$ there is a unique $\mu$ of size $l = n - \lambda_1$, given by $\mu_i' = \lambda_i'$, so that $\lambda/\mu$ forms a horizontal strip. As $\mu$ might not have desarrangement tableaux, it might not correspond to an eigenvalue. Then for this unique $\mu$, $\diag(\lambda) - \diag(\mu) = \sum_{j=1}^{\lambda_1} j -| \lambda_j'| = {n-l + 1 \choose 2} - n$. 
\begin{align*} \eig(\lambda/\mu) &= \frac{1}{n^2}\left( {n+1 \choose 2} - {l + 1 \choose 2} + {n-l + 1 \choose 2} - n \right) \\
& = 1 -\frac{ (n+1)l}{n^2}\\
& \leq 1 - \frac{l}{n}
\end{align*}

For $k \geq 1$ and $l$ fixed, over all horizontal strips $\lambda/\mu$ with $\lambda_1 = n-l$ and $|\mu| = l+k$ the eigenvalues are maximized when the boxes in the skew shape are as far to the right and up as possible. When $\lambda = [n-l, 1^l]$, $\mu = [k,1^l]$, $\lambda/\mu$ is a horizontal strip and the $n-l-k$ boxes of $\lambda/\mu$ are all in the first row, achieving this maximum, with
\[ \diag([n-1,1^l]) - \diag([k,1^l]) = (n-l-1) + \ldots + (k+1 -1) = {n -l \choose 2} - {k \choose 2}.\] 
Simplifying, 
\[\binom{n+1}{2} - \binom{k+l+1}{2} + \binom{n-l}{2} - \binom{k }{ 2} = n^2\left( 1 - \frac{l}{n} - \frac{k^2 + kl}{n^2} \right).\]

\end{proof}
\subsection{Bounding multiplicities}\label{boundmult}
In this section, we introduce bounds on the multiplicities of the eigenvalues. Theorem \ref{DieSal} says that  the eigenvalue $\eig(\lambda/\mu)$ occurs with multiplicity $d_\lambda d^\mu$ where $d_\lambda$ is the number of standard Young tableaux of shape $\lambda$ and $d^\mu$ is the number of desarrangement tableaux of shape $\mu$.
Proposition 9.4 of Reiner, Saliola, Welker \cite{RSW} gives a bijection between tableaux of shape $\lambda$ and the desarrangement tableaux of the shapes $\mu$, for which $\lambda/\mu$ is a horizontal strip. As consequence of this, $\sum_\mu d^\mu = d_\lambda$. The total multiplicity of eigenvalues corresponding to a $(\lambda,\cdot)$ pair is thus $d_\lambda^2$. 

To bound the $d_{\lambda},$ we will make use of the following proposition which is an outcome of Corollary $2$ of Diaconis and Shahshahani \cite{DiaSha}.
\begin{proposition}[Diaconis-Shahshahani]\label{mult}
We have $d_\lambda \leq {n \choose n-\lambda_1} d_{\lambda/\lambda_1}$ and
\[ \sum_{\lambda: n- \lambda_1 = l} d_\lambda^2 \leq {n \choose l}\frac{n!}{(n-l)!}.\]
\end{proposition}
To bound the $d^{\mu},$ we will need one further concept from the combinatorics of tableaux. The {\it jeu de taquin} is a transformation of a skew standard Young tableau that gives a skew standard Young tableau of a different shape. Given a tableau of a skew shape $\lambda/\mu$ and an empty cell $c$ that can be added to the skew shape (meaning that the new tableau is still of skew shape), do moves as follows. The fact that $c$ can be added to the skew shape means that there are two cells of $\lambda/\mu$ adjacent to $c$. Move the entry into $c$ that maintains the property that entries are increasing along rows and columns. 

Let the interior of $\lambda / \mu$ consist of the boxes of $\mu$. Let the exterior of a tableau of $\lambda / \mu$ consist of the boxes that, if added to $\lambda$, the resulting diagram is still a partition. If the empty cell starts on the interior of the skew shape, it moves to the exterior, and vice versa. We will be referring to these moves as jeu de taquin slides.

\begin{example}\label{jeu_zero}
Starting with the tableau $\tikztableausmall{{    X,    X, 2,5 },{   1,3, 4}}$, jeu de taquin on the cell $(1,2)$ gives:

\[ \tikztableausmall{{    X,    , 2,5 },{   1,3, 4}} \rightarrow \tikztableausmall{{    X,    2, ,5 },{   1,3, 4}} \rightarrow \tikztableausmall{{    X,    2,4 ,5 },{   1,3, }} \rightarrow \tikztableausmall{{    X,   2 , 4 ,5 },{   1,3}} \]

\end{example}

\begin{example}
Starting with the same tableau, jeu de taquin on the cell $(2,4)$ gives:
\[ \tikztableausmall{{    X,    X, 2,5 },{   1,3, 4, \null}} \rightarrow \tikztableausmall{{    X,    X, 2, \null },{   1,3, 4, 5}}  \rightarrow \tikztableausmall{{    X,    X, \null, 2 },{   1,3, 4,5}} \rightarrow \tikztableausmall{{    X,    X, X,2 },{   1,3, 4,5}} \]
\end{example}

The following proposition will be proven at the end of this section.
\begin{proposition}\label{moremult}
For $\lambda$ a partition of $n$, let $l = n-\lambda_1$ and fix $k$. We have that
 \[\sum_{\mu: |\mu|=l +k} d^\mu \leq {l + k \choose l-1}d_{\lambda/\lambda_1,}\] 
where we are summing over all partitions $\mu$ with $|\mu| = n- \lambda_1 + k$ for which $\lambda/\mu$ is a horizontal strip.
\end{proposition}

Reiner, Saliola and Welker proved that there is a bijection between tableaux of shape $\lambda$ and all desarrangment tableaux of the shapes $\mu$ for which $\lambda/\mu$ is a horizontal strip. The bijection can be described as: starting with $\hat{Q}$ a desarrangement tableau of shape $\mu$, with $\lambda/\mu$ a horizontal strip, append the cells of $\lambda/\mu$  as empty cells. Perform jeu de taquin slides into these empty cells in order from the leftmost to rightmost cell of $\lambda/\mu$. As a consequence of $\lambda/\mu$ being a horizontal strip, the resulting tableau has the first $|\lambda/\mu|$ cells of the first row of $\lambda$ empty. Filling the first row with $1,\ldots,|\mu|$ and incrementing the rest by $|\lambda/\mu|$ gives a standard tableau of shape $\lambda$.
 
 \begin{example}\label{jeu_one}
If $\lambda=\young(~~~~,~~~,~~)$, $\mu= \young(~~~,~~~,~)$ and $\hat{Q}= \young(134,267,5)$. Then the jeu de taquin gives
 $$ \young(134~,267,5~) \rightarrow \young(134~,2~7,56) \rightarrow \young(1~4~,237,56)\rightarrow \young(~14~,237,56)\rightarrow \young(~1~4,237,56)\rightarrow \young(~~14,237,56)$$
Reassigning the values, we get the following tableau:
$$ \young(1236,459,78)$$
\end{example}

Notice that the process described, results in a skew diagram that is missing labels for the cells only in the first row, exactly because $\lambda/\mu$ is a horizontal strip. To do the inverse of the above function, given a tableau, we want to retrieve a skew tableau that is missing labels for cells only in the first row. Reiner, Saliola and Welker describe the inverse process in the proof of Proposition 9.4 of \cite{RSW}. For completion, we are re-writing the the statement of their proposition, as applied to our case.
\begin{proposition}[Proposition 9.4, \cite{RSW}]\label{rsw14}
Given a tableau, find the unique $a,b$ so that the initial sequential entries in the first row are $1,2,\ldots,a$, the first column begins $1,a+1,\ldots,a+b$, and the cell containing $a+b+1$ is not in the first column. If $b$ is odd, remove $1,\ldots,a-1$, perform jeu de taquin slides from the rightmost box to the leftmost one, and subtract $a-1$ from all entries. If $b$ is even, remove $1,\ldots,a$, perform jeu de taquin slides from the rightmost box to the leftmost one, and subtract $a$ from all entries. This is the inverse function of the function directly described above Example \ref{jeu_one}.
\end{proposition}
\begin{example}
Looking at 
$$ \young(1236,459,78)$$
we see that $a=3$ and $b=1$. Therefore, we delete the entries $1$ and $2$ and we perform the jeu de taquin as follows:
$$ \young(~~36,459,78) \rightarrow \young(~3~6,459,78) \rightarrow \young(~36~,459,78) \rightarrow \young(3~6~,459,78) \rightarrow \young(356~,4~9,78) \rightarrow \young(356~,489,7~) $$
Then we subtract $a=2$ from the entries and we get exactly 
$\hat{Q}= \young(134,267,5)$, which we started with in Example \ref{jeu_one}.
\end{example}


\begin{proof}[Proof of Proposition \ref{moremult}]

We will bound the number of tableaux of shape $\lambda$ that give rise to a partition of size $n-\lambda_1 + k$ after performing the function of Proposition \ref{rsw14}. Let $l = n-\lambda_1$. The number of blocks removed must be $n-(l + k) = n-l -k$. We will split this into cases based on the parity of $b$.

For $b$ odd, this means the $a$ described in Proposition \ref{rsw14} is given by $a = \lambda_1-k+1$. We need to count tableaux with first row starting $1,\ldots,a$ and first column starting $1,a+1,\ldots,a+b$. An upper bound for the number of such tableaux can be computed by selecting any $l-b$ elements of $a+b+1,\ldots,n$ as the entries that appear below the first row in any tableaux of shape $\lambda/[\lambda_1,1^b]$, and the remaining entries in order in the remaining positions of the first row. This is at most \[{n - a- b \choose l-b} d_{\lambda/\lambda_1} ={l +k -b-1 \choose l-b} d_{\lambda/\lambda_1} \leq {l +k -b \choose l-b} d_{\lambda/\lambda_1} .\]

For $b$ even, letting $a = \lambda_1 -k$, we need to count tableaux with first row $1,\ldots,a$ and first column starting $1,a+1,\ldots,a+b$. Using the same approximation technique, the number of such tableaux is at most:
\[ {n -a- b \choose l-b} d_{\lambda/\lambda_1} = {l +k -b \choose l-b} d_{\lambda/\lambda_1}. \]

Note that the final form of the bound is the same for both $b$ odd and even, so summing over values of $b$ of the binomial coefficients using the combinatorial identity $\sum_{i=0}^m {n+i \choose i} = {n+m+1 \choose m}$ is:

\[ \sum_{b=1}^l {l + k -b \choose l-b} \leq {l+k \choose l-1}.\]
Therefore,  
\[\sum_{|\mu|=n-\lambda_1 +k} d^\mu  \leq \sum_{b=1}^l {l + k -b \choose l-b}d_{\lambda/\lambda_1} \leq {l+k \choose l-1} d_{\lambda/\lambda_1},\] 
which finishes the proof of the proposition.
\end{proof}

\section{$\frac{3}{4}n \log{n}- \frac{1}{4}\log\log n$ upper bound}\label{prf}
In this section, we present the proof of Theorem \ref{rtr}.
To achieve the upper bound of $\frac{3}{4}n \log n - \frac{1}{4}n \log\log{n} +cn$, it is necessary to take advantage of how $\eig(\lambda/\mu)$ decreases and $d^\mu$ increases as more blocks are added to $\mu$. Only the eigenvalues for partitions $\mu$ of size at most $n-\lambda_1+\sqrt{(n-\lambda_1)n}$ will significantly contribute. This will reduce the contribution of $d_\lambda d^\mu$ to be considered from order $n^{2(n-\lambda_1)}$ to $n^{\frac{3}{2}(n-\lambda_1)}$ leading to an upper bound of $t = \frac{3}{4}n\log(n) +cn$. To further refine the upper bound to show cutoff with window, one must take advantage of the decay in this initial group of eigenvalues. We show this is closely approximated by passing to a Gaussian integral.

To illustrate this, we will consider as a special case the terms for $\lambda = [n-1,1]$, $\mu= [k,1]$. The largest eigenvalue corresponds to $k=1$ with multiplicity $n-1$. If the mixing time were given only by the time it takes for $d_{[n-1,1]}d^{[1,1]} (\eig({[n-1,1]/[1,1]}))^{2t} = (n-1)\left( 1- \frac{n+2}{n^2}\right)^{2t}$ to be exponentially small, the mixing time would be $\frac{1}{2}n\log n $. Some of the spectrum of this walk is closely bunched around this eigenvalue, and so still contributes to the mixing time. If it were entirely bunched around this eigenvalue, the mixing time would be $n\log n$. The fact that the spectrum is spread out gives rise to the odd coefficient of $\frac{3}{4}$. 
\begin{proposition}\label{largesteig}
For $t = \frac{3}{4}n\log n - \frac{1}{4}n\log\log{n}+ cn,$ we have that
$$ \sum_{k=1}^{n-1} d_{[n-1,1]}d^{[k,1]} (\eig([n-1,1]/[k,1]))^{2t} = \sum_{k=1}^{n-1} (n-1)\left( 1- \frac{n+k^2+k}{n^2}\right)^{2t} \leq  e^{-2c},$$
where $c>0$.
\end{proposition}

\begin{proof}
The tableaux of shape $[n-1,1]$ have $(1,1)$ entry $1$ and have $n-2$ further increasing entries in the first row, and so are determined by which of $2,\ldots,n$ is in $(2,1)$, so $d_{[n-1,1]}= n-1$. The number of desarrangement tableaux of shape $[k,1]$ is $1$ since the first column must be sequential up to an even number. The only option for the $(2,1)$ entry is $2$, so $d^{[k,1]} = 1$.

Computing these eigenvalues from equation \eqref{eigen} gives, $\eig([n-1,1]/[k,1]) = 1 - \frac{n+k^2 + k}{n^2}$. To bound the contribution of these eigenvalues to the mixing time, we will pass from a sum to a Gaussian integral,
\begin{align} \sum_{k=1}^{n-1} (n-1)\left( 1- \frac{n+k^2+k}{n^2}\right)^{2t} 
& = (n-1)\left(1 - \frac{1}{n}\right)^{2t}\sum_{k=1}^{n-1}\left( 1- \frac{k^2+k}{n^2-n}\right)^{2t} \nonumber\\ 
& \label{in} \leq e^{-2t/n + \log{n}} \sum_{k=1}^{n} e^{-2t \frac{k^2+k}{n^2-n}}\\
& \label{riem} \leq e^{-\frac{1}{2}\log{n} + \frac{1}{2}\log\log{n} - 2c} \int_{0}^n e^{-2t \frac{x^2 + x}{n^2-n}} dx\\
& \label{sub}\leq e^{-2c}\sqrt{\frac{\log{n}}{n}}\int_{0}^n e^{-2t \frac{x^2}{n^2-n}} dx\\
& \label{sub2}\leq e^{-2c} \sqrt{\frac{\log{n}}{n}} \sqrt{\frac{n^2-n}{2t}}\int_0^\infty e^{-x^2} dx\\
& \label{l} \leq e^{-2c}
\end{align}
Inequality \eqref{in} doubly utilizes the inequality $1-x \leq e^{-x}$. Next, since the function $e^{-C(k^2+k)}$ is decreasing and we are integrating from $0$ to $n$ but summing from $1$ to $n-1$, the equivalent Riemman sum falls below the curve, and the \eqref{riem} holds. A substitution to remove the constant in the integral, and increase in the integral being integrated over gives \eqref{sub} and \eqref{sub2}. Finally, both the product of square roots and the integral are bounded by $1$, which gives \eqref{l}.
\end{proof}

For all other $\lambda$, the $d^{\mu}$ vary, and so understanding how they are distributed is essential in bounding $\sum_\mu d_\lambda d^\mu (\eig(\lambda/\mu))^{2t}$. This is where we will use the bounds of sections \ref{eigenvalues} and \ref{boundmult}. We are now ready to prove our main result:

\begin{reptheorem}{rtr}
For the random to random card shuffle, if $t=  \frac{3n}{4} \log n - \frac{1}{4}n\log\log{n} +c n$ with $c \geq 2$ and $n$ sufficiently large, then
\begin{equation*}
||P_{id}^{*t}-\pi ||_{\TV} \leq e^{-c}.
\end{equation*}
\end{reptheorem}
\begin{proof}
The $l^2$ bound of Lemma \ref{l2} gives that
\begin{align}
4||P^t - \pi ||^2_{\TV} \leq \left|\left| \frac{P^t(x,\cdot)}{\pi(\cdot)} -1 \right|\right|^2_2 &\leq \sum_{\lambda,\mu} d_{\lambda} d^{\mu} \eig(\lambda/\mu)^{2t} \nonumber\\
\label{eigb} & \leq \sum_{l=1}^{n-1}\sum_{\lambda: \lambda_1 = n-l} d_\lambda \sum_{k=0}^{n-l} \left(\sum_{\mu: |\mu| =l+k} d^\mu\right) \left(1 - \frac{l}{n} - \frac{k^2 +kl}{n^2} \right)^{2t}  \\
\label{m}
& \leq \sum_{l=1}^{n-1}\sum_{\lambda: \lambda_1 = n-l} {n \choose l} d_{\lambda/\lambda_1} \sum_{k=0}^{n-l} {k+l \choose l-1} d_{\lambda/\lambda_1}\left(1 - \frac{l}{n} - \frac{k^2 +kl}{n^2} \right)^{2t} \\ 
\label{denom}
&\leq \sum_{l=1}^{n-1}\sum_{\lambda: \lambda_1 = n-l} {n \choose l} d_{\lambda/\lambda_1}^2 \left(1 - \frac{l}{n}\right)^{2t}\sum_{k=0}^{n-l} {k+l \choose l-1}\left(1 -\frac{k^2 +kl}{n^2} \right)^{2t}  \\
& \label{initial}\leq \sum_{l=1}^{n-1} n^l e^{\frac{-2tl}{n}} \sum_{k=0}^{n-l} {k+l \choose l-1} e^{-2t \frac{k^2 +kl}{n^2}}
\end{align}
where in \eqref{eigb} we use Proposition \ref{ieb}. In equation \eqref{denom}, we use that $1-x-y \leq (1-x)(1-y)$. In \eqref{m} we use Proposition \ref{moremult}, and in \eqref{initial} we use Proposition \ref{mult}, plus as a result of the Robinson-Schensted-Knuth correspondence $\sum_{\lambda: \lambda_1 = n-l} d_{\lambda/\lambda_1}^2= l!$, that $l!\binom{n}{l} \leq n^l$, and two applications of $1-x \leq e^{-x}$. 

For the terms with $l \geq n/2$, $\eig(\lambda,\mu) \leq 1 - \frac{l}{n} \leq 1/2$, and, we see the crude bound, 
\begin{equation}\label{crude}
\sum_{l\geq n/2}\sum_{\lambda: \lambda_1=n-l} d_\lambda^2 \eig(\lambda,\mu_\lambda)^{2t} \leq n! \cdot 2^{-2t}\leq e^{n\log{n} - \frac{3}{2}\log(2) n \log{n} + \frac{1}{2}\log(2)n\log\log{n} -2\log(2)cn} \leq e^{-2c},
\end{equation}
where the last two inequalities follow from $n! \leq n^n$, using that for $c \geq 2$ and $n \geq 2/(2\log{2}-1)$ we have $2\log(2)cn \geq n + 2c$ and that for $n>1$,
\[ n\log{n} - \frac{3}{2}\log(2) n \log{n} + \frac{1}{2}\log(2)n\log\log{n} -n <0.\]  
Interestingly, these small eigenvalues are decaying so quickly, the arguments for the transposition walk of Diaconis and Shahshahani \cite{DiaSha} in their outer and mid zones show for these eigenvalues the much smaller $t = \frac{1}{2}n\log{n} + cn$ would suffice.  

Proposition \ref{largesteig} covers the case $l=1$ which is attained by $\lambda = [n-1,1]$. The majority of the remainder of the proof consists of bounding for $2 \leq l \leq n/2$ the sum over $k$ in equation \eqref{initial}. We will show the sum is at most $e^{2l}\left(\frac{n^2}{2t}\right)^{l/2}$ by passing to a Gaussian integral after completing the square in the exponent. Once we have shown this, the proof resumes at \eqref{flourish}.

\begin{align} 
&  \sum_{k=0}^{n-l} \binom{l+k}{l-1} e^{-\frac{2t(k^2 +kl)}{n^2}} \nonumber \\ 
\label{bin} \leq & \frac{2^{l-1}}{(l-1)!}e^{tl^2/(2n^2)}\sum_{k=0}^{n-l} (k+l/2)^{l-1} e^{-\frac{2t(k+l/2)^2}{n^2}}\\
\label{bound}\leq & \frac{2^{l-1}}{(l-1)!}e^{tl^2/(2n^2)}\left(\int_{0}^\infty  (k+l/2)^{l-1} e^{-\frac{2t(k+l/2)^2}{n^2-nl}} dk + \sup_{k \geq 0}  (k+l/2)^{l-1} e^{-\frac{2t(k+l/2)^2}{n^2}}\right)\\
\label{Gam} \leq & \frac{2^{l-1}}{(l-1)!}e^{tl^2/(2n^2)}\left(  \Gamma(l/2) \left(\frac{n^2}{2t} \right)^{l/2} +  \left(\frac{(l-1)/2}{e}\right)^{(l-1)/2}\left( \frac{n^2}{2t}\right)^{(l-1)/2}\right)\\
\label{cont} \leq & \frac{2^{l-1}}{(l-1)!}e^{tl^2/(2n^2)}\left(  \Gamma(l/2)  +  \left(\frac{(l-1)/2}{e}\right)^{(l-1)/2} \right) \left(\frac{n^2}{2t} \right)^{l/2}
\end{align}

For \eqref{bin}, we  note that since $k \geq 0$ and $l\geq 2$, $\binom{l+k}{l-1} \leq \frac{(l+k)^{l-1}}{(l-1)!} \leq \frac{2^{l-1} (l/2+k)^{l-1}}{(l-1)!}$. We also complete the square in the exponent $e^{-\frac{2t(k^2 +kl)}{n^2}}$ and this is why we end up with the term $e^{tl^2/(2n^2)}$ outside the sum in \eqref{bin}. 

For \eqref{bound}, we are addressing the error introduced by passing from a sum to an integral. We are summing $g(k) = (k+l/2)^{l-1} e^{-2t(k+l/2)^2/n^2}$ from $k=0$ to $n-l$. The function $g(x)$ is increasing to where $(x+l/2)^2 = \frac{l-1}{2}\frac{n^2}{2t}$ and then decreasing. When the function is increasing, the right Reimann sum is bounded by the integral, and similarly, when the function is decreasing the left Reimann sum is bounded by the integral. We just need to add the value of the function when the transition happens. Therefore, $\sum_{x=1}^{n-l} g(x) \leq \int_0^\infty g(x) dx + \sup_{x \geq 0} g(x)$. The integral and supremum will be of remarkably similar order. The supremum is $\left(\frac{(l-1)/2}{e}\right)^{(l-1)/2}\left( \frac{n^2}{2t}\right)^{(l-1)/2},$ which is used in \eqref{Gam}.

For the rest of  \eqref{Gam}, we recall the definition of the Gamma function $$\Gamma(s)=\int_0^\infty x^{s-1} e^{- x} dx.$$
Performing the substitution $u=a x^2$, we get that  $\int_0^\infty x^c e^{-a x^2} dx = \frac{\Gamma((c+1)/2)}{2a^{(c+1)/2}}$. This leads to the equality $\int_{0}^\infty  (k+l/2)^{l-1} e^{-\frac{2t(k+l/2)^2}{n^2}} dk =\Gamma(l/2) \left(\frac{n^2}{2t} \right)^{l/2}$.

It remains to show that $\eqref{cont}/\left(\frac{n^2}{2t} \right)^{l/2} \leq e^{2l}$. We have that
\begin{align}
\eqref{cont}/\left(\frac{n^2}{2t} \right)^{l/2}&= \frac{2^{l-1}}{(l-1)!}e^{tl^2/(2n^2)}\left(  \Gamma(l/2)  +  \left(\frac{(l-1)/2}{e}\right)^{(l-1)/2} \right) \cr
&\label{Stir}\leq 2^{l-1} e^{tl^2/(2n^2)}\left( \frac{2e\sqrt{\pi\frac{(l-1)}{2}} \left(\frac{l-1}{2e}\right)^{\frac{l-1}{2}}}{\sqrt{2\pi (l-1)}\left(\frac{l-1}{e}\right)^{l-1}} \right)\\
&\label{t}\leq 2^{l-1} e^{\frac{l-1}{2}\log (l-1)  + \frac{1}{2}(l-1) + \frac{l-1}{2}\log{2}} \left( \frac{ e^{(l+1)/2}}{(2(l-1))^{\frac{l-1}{2}}} \right)\\
& = 2^{l-1} e^{\frac{1}{2}(l-1) + \frac{l-1}{2}\log{2}}  e(e/2)^{(l-1)/2}\\
& = 2^{l-1}e^{l}\\
& \label{g} \leq e^{2l}
\end{align}

For \eqref{Stir}, we have that $\Gamma(l/2) +  \left(\frac{(l-1)/2}{e}\right)^{(l-1)/2}\leq 2e\sqrt{\pi(l-1)/2} \left((l-1)/(2e)\right)^{(l-1)/2}$. This is because of the following properties of the Gamma function, that can be found in Lemma 1.2 and Theorem 1.4 of \cite{Stein}:
$\Gamma(1/2)= \sqrt{\pi}$, $\Gamma(1) =1$, and $\Gamma(n) = (n-1)\Gamma(n-1),$ for any real $n$. Iterating, we can bound $\Gamma(l/2) \leq \sqrt{\pi}\left((l-1)/2\right)!$ for $l$ odd and $(\frac{l}{2}-1)!$ for $l$ even. Using the following form of Stirling's approximation \cite{StirlingRobbins} $$\sqrt{2\pi n}(n/e)^n \leq n! \leq e \sqrt{n} (n/e)^n,$$ 
which works for all values of $n$, we see that $(l-1)! \geq \sqrt{2\pi (l-1)}((l-1)/e)^{l-1} $ and $\Gamma(l/2) +  \left(\frac{(l-1)/2}{e}\right)^{(l-1)/2}\leq 2e\sqrt{\pi(l-1)/2} \left((l-1)/(2e)\right)^{(l-1)/2}$, as desired. 

For \eqref{t}, we want to prove that $tl^2/(2n(n-l)) \leq \frac{l-1}{2}\log (l-1)  + \frac{1}{2}(l-1) + \frac{l-1}{2}\log{2}.$ Recall that $l \leq n/2$. Letting $t=  \frac{3n}{4} \log n - \frac{1}{4}n\log\log{n} +c n$, we get that 
\begin{equation}\label{a}
tl^2/(2n^2) \leq \frac{3l^2\log{n}}{8n}.
\end{equation}
Since $\log{x}/x$ is decreasing for $x \geq e$, we have that $\frac{\log{n}}{n} \leq \frac{\log{l}}{l}$, and
\begin{equation}\label{b}
\frac{3l^2\log{n}}{8n} \leq \frac{l}{2}\log{l}.
\end{equation}
Finally, using that $l\geq 2$, 
\begin{equation}\label{c}
\frac{l}{2}\log{l} = \frac{l-1}{2}\log (l-1) + \frac{1}{2}\log{l} + \frac{l-1}{2}\log{\frac{l}{l-1}}\leq  \frac{l-1}{2}\log (l-1)  + \frac{1}{2}(l-1) + \frac{l-1}{2}\log{2}.
\end{equation}
Equations \eqref{a}, \eqref{b} and \eqref{c} give exactly that $tl^2/(2n^2) \leq \frac{l-1}{2}\log (l-1)  + \frac{1}{2}(l-1) + \frac{l-1}{2}\log{2}$. The rest of \eqref{t} is simplifying the terms inside the last parenthesis. 

Finally, \eqref{initial}, \eqref{crude}, \eqref{cont}, \eqref{g}, combined with the fact that for $n \geq 1$, $t \geq \frac{1}{2}n\log{n}$, and so $e^{-l/2\log(2t)} \leq e^{-l/2\log\left(n\log{n}\right)}$, give that

\begin{align} 
\label{flourish} 4||P^t - \pi ||^2_{\TV} &\leq \left|\left| \frac{P^t(x,\cdot)}{\pi(\cdot)} -1 \right|\right|^2_2 \leq e^{-2c} + e^{-2c}+ \sum_{l=2}^{n/2}n^l e^{-2tl/n} e^{2l}\left(\frac{n^2}{2t} \right)^{l/2} \\
\nonumber \leq &2e^{-2c}+\sum_{l=2}^{n/2}e^{2l\log{n}+2l - 2tl/n - l/2\log{2t}}\\
\nonumber \leq &2e^{-2c}+\sum_{l=2}^{n/2}e^{2l\log{n}+2l - 3l/2\log{n} + l/2\log\log{n} - 2cl -l/2\log{n} - l/2\log\log{n}}\\
\nonumber \leq &2e^{-2c}+\sum_{l=2}^{n/2}e^{-2(c-1)l}\\
\nonumber \leq & 4e^{-2c},
\end{align}
since $\sum_{l=2}^{n/2} e^{-2(c-1)l} \leq \frac{e^{-4(c-1)}}{1 - e^{-2(c-1)}} \leq 2e^{-2c}$ for $c \geq 2$. Therefore, $4\left|\left| P^t(x,\cdot)-\pi(\cdot) \right|\right|_{\TV}^2 \leq 4e^{-2c}$ and $||\left|\left| P^t(x,\cdot)-\pi(\cdot) \right|\right|_{\TV} \leq e^{-c}$. This completes the proof of Theorem \ref{rtr}.
\end{proof}

\section{Random to random shuffle with repeated cards}\label{generalization}
In this section, we briefly examine the case where the $n$ cards of the deck are not necessarily distinct. Let $\nu=[\nu_1,\ldots, \nu_m]$ be a partition of $n$. Then we can perform the random to random shuffle on a deck which has $\nu_i$ cards of type $i$, where $i=1,\ldots ,m$. We then say that the deck has evaluation $\nu$.
 \begin{reptheorem}{word}
For the random to random shuffle on a deck with evaluation $\nu=[\nu_1,\ldots, \nu_m]$ with $\nu_1 \neq n$, we have that 
\begin{enumerate}
\item For $t=  \frac{3}{4} n\log n - \frac{1}{4}n\log\log{n} +cn$ and $n$ sufficiently large, then 
\begin{equation*}
||P_{id}^{*t}-\pi ||_{TV} \leq \frac{1}{2}\left|\left| \frac{P_{id}^{*t}}{\pi} -1 \right|\right|_2  \leq  e^{-c}.
\end{equation*}
where $c>2$.
\item For $t= \frac{n}{4} \log m + \frac{n}{8} \log n -cn$, then 
\begin{equation*}
\left|\left|\frac{P_{id}^{*t}}{\pi}-1 \right|\right|_{2} \geq \frac{1}{2}e^{c}.
\end{equation*}
where $c>0$.
\end{enumerate}
\end{reptheorem}
The proof is similar to the proof of Theorem \ref{rtr}, therefore we will quickly review what we need for this case. Dieker and Saliola \cite{DSaliola} proved that the eigenvalues of this walk are given by 
\begin{equation}\label{eigen2}
\eig( \lambda/ \mu) = \frac{1}{n^2} \left( {n+1 \choose 2} - {|\mu| + 1 \choose 2} + \Diag(\lambda) - \Diag(\mu) \right)
\end{equation}
with the restriction that $\lambda \trianglerighteq \nu$ which means that 
$$\lambda_1+\ldots +\lambda_i \geq \nu_1 +\ldots +\nu_i$$
for every $i=1,\ldots,m$. The multiplicity of such an eigenvalue is $$\sum_{\substack{\lambda/\mu \mbox{ is a horizontal strip}\\ \lambda \trianglerighteq \nu } } K_{\lambda,\nu} d^{\mu}$$
where $K_{\lambda,\nu} $ is the number of semistandard 
Young tableaux (SSYT) of shape $\lambda$ and evaluation $\nu$ and $d^{\mu}$ is the number of desarrangement tableaux of shape $\mu$, just like before.

\begin{proof}[Proof of Theorem \ref{word}]
The first part of Theorem \ref{word} is easy to prove because $K_{\lambda,\nu} \leq d_{\lambda}$. Therefore, imitating the proof of Theorem \ref{rtr} we get that for $t=  \frac{3}{4}n \log n - \frac{1}{4}n\log\log{n} +cn$, then 
\begin{equation*}
\frac{1}{2}\left|\left|\frac{P_{id}^{*t}}{\pi}-1 \right|\right|_{2} \leq  e^{-c}.
\end{equation*}. 

For the second part we only need to check the case where $\lambda= [n-1,1]$, for which we have that $\lambda \unrhd \nu$. In this case, $K_{\lambda,\nu} =m-1$ and $d^{\mu}=1$. Just like in Proposition \ref{largesteig}, if $t= \frac{n}{4} \log m + \frac{n}{8} \log n -cn$, then
\begin{align*} \sum_{k=1}^{n-1} (m-1)\left( 1- \frac{n+k^2+k}{n^2}\right)^{2t} \geq \sum_{k\leq \sqrt{n}-1}(m-1) \left(1- \frac{2}{n}\right)^{2t} =\left(\sqrt{n}-1\right)(m -1)\left(1- \frac{2}{n}\right)^{2t} \geq \frac{1}{2} e^{2c},
\end{align*}
using that for $x \leq \frac{1}{2}$, $1-x \geq e^{-x(1+x)}$ and $c>0$.
\end{proof}

\section{An application to a time inhomogeneous card shuffle}\label{lastsection}
Recently, there has been significant interest in time inhomogeneous chains, in which the generators of the walk are not the same at each step. For example, the card-cyclic to random shuffle at time $t$, the card with the label $t \mod {n+1}$ is removed from the deck and gets inserted it to a uniformly random position of the deck.  It has been studied by Pinsky \cite{Pinsky}, Morris, Ning, Peres \cite{MNP} and Saloff-Coste and  Z{\'u}{\~n}iga \cite{SZ1}, \cite{SZ2}, \cite{SZ3}, \cite{SZ4}. Since the eigenvalues of the random to random shuffle correspond to singular values for the random to top card shuffle, it is possible to bound the $\ell^2$ mixing time of card-cyclic to random in terms of the $\ell^2$ mixing time of the random to random shuffle. 

The current best upper bound for the mixing time of card-cyclic to random is $4 n \log n$ and it is due to Saloff-Coste and Z{\'u}{\~n}iga \cite{SZ1}. Theorem \ref{rtr} gives the following improvement on the mixing time of card-cyclic to random.
\begin{corollary}
For $t=\frac{3}{2} n \log n + cn$, where $c>0$,  and $x \in S_n$, we have that
$$ 4\lVert Q_{x}^t-U \rVert^2_{T.V.} \leq e^{-c} $$
where $Q$ is the transition matrix of cyclic to random insertions and $U$ is the uniform measure.
\end{corollary}
\begin{proof}
Theorem $4.6$ of \cite{SZ1} says that
$$\lVert Q_{x}^{2t}-U \rVert_{2} \leq \lVert P_{x}^{*t}-U \rVert_{2}$$
where $P$ is the distribution of the generators of random to random. Theorem \ref{rtr} finishes the proof.
\end{proof}

{\it Acknowledgements.} We would like to thank Franco Saliola for several enlightening conversations during the course of the project. We would also like to thank Ton Dieker for his useful comments. We would like to thank a referee for the very short proof of Lemma \ref{positivity}.

 \bibliographystyle{plain}
\bibliography{changes}

\begin{thebibliography}{10}

\bibitem{CS}
Guan-Yu Chen and Laurent Saloff-Coste.
\newblock The cutoff phenomenon for ergodic {M}arkov processes.
\newblock {\em Electron. J. Probab.}, 13:no. 3, 26--78, 2008.

\bibitem{DiaconisSurvey}
P~Diaconis.
\newblock The cutoff phenomenon in finite markov chains.
\newblock 93:1659--64, 03 1996.

\bibitem{DiaConj}
Persi Diaconis.
\newblock Mathematical developments from the analysis of riffle shuffling.
\newblock In {\em Groups, combinatorics \& geometry ({D}urham, 2001)}, pages
  73--97. World Sci. Publ., River Edge, NJ, 2003.

\bibitem{DS}
Persi Diaconis and Laurent Saloff-Coste.
\newblock Comparison techniques for random walk on finite groups.
\newblock {\em Ann. Probab.}, 21(4):2131--2156, 1993.

\bibitem{DiaSha}
Persi Diaconis and Mehrdad Shahshahani.
\newblock Generating a random permutation with random transpositions.
\newblock {\em Z. Wahrsch. Verw. Gebiete}, 57(2):159--179, 1981.

\bibitem{DSaliola}
A.~B. Dieker and F.~V. Saliola.
\newblock Spectral analysis of random-to-random {M}arkov chains.
\newblock {\em Adv. Math.}, 323:427--485, 2018.

\bibitem{LPW}
David~A. Levin, Yuval Peres, and Elizabeth~L. Wilmer.
\newblock {\em Markov chains and mixing times}.
\newblock American Mathematical Society, Providence, RI, 2009.
\newblock With a chapter by James G. Propp and David B. Wilson.

\bibitem{MNP}
Ben Morris, Weiyang Ning, and Yuval Peres.
\newblock Mixing time of the card-cyclic-to-random shuffle.
\newblock {\em Ann. Appl. Probab.}, 24(5):1835--1849, 2014.

\bibitem{Pinsky}
Ross~G. Pinsky.
\newblock Probabilistic and combinatorial aspects of the card-cyclic to random
  insertion shuffle.
\newblock {\em Random Structures Algorithms}, 46(2):362--390, 2015.

\bibitem{MorrisQin}
Chuan Qin and Ben Morris.
\newblock Improved bounds for the mixing time of the random-to-random shuffle.
\newblock {\em Electron. Commun. Probab.}, 22:Paper No. 22, 7, 2017.

\bibitem{RSW}
Victor Reiner, Franco Saliola, and Volkmar Welker.
\newblock Spectra of symmetrized shuffling operators.
\newblock {\em Mem. Amer. Math. Soc.}, 228(1072):vi+109, 2014.

\bibitem{StirlingRobbins}
Herbert Robbins.
\newblock A remark on stirling's formula.
\newblock {\em The American Mathematical Monthly}, 62(1):26--29, 1955.

\bibitem{SZ1}
L.~Saloff-Coste and J.~Z\'u\~niga.
\newblock Convergence of some time inhomogeneous {M}arkov chains via spectral
  techniques.
\newblock {\em Stochastic Process. Appl.}, 117(8):961--979, 2007.

\bibitem{SZ2}
L.~Saloff-Coste and J.~Z\'u\~niga.
\newblock Merging for time inhomogeneous finite {M}arkov chains. {I}.
  {S}ingular values and stability.
\newblock {\em Electron. J. Probab.}, 14:1456--1494, 2009.

\bibitem{SZ3}
L.~Saloff-Coste and J.~Z\'u\~niga.
\newblock Time inhomogeneous {M}arkov chains with wave-like behavior.
\newblock {\em Ann. Appl. Probab.}, 20(5):1831--1853, 2010.

\bibitem{LZ}
L.~Saloff-Coste and J.~Z{\'u}{\~n}iga.
\newblock Refined estimates for some basic random walks on the symmetric and
  alternating groups.
\newblock {\em ALEA Lat. Am. J. Probab. Math. Stat.}, 4:359--392, 2008.

\bibitem{SZ4}
Laurent Saloff-Coste and Jessica Z\'u\~niga.
\newblock Merging and stability for time inhomogeneous finite {M}arkov chains.
\newblock In {\em Surveys in stochastic processes}, EMS Ser. Congr. Rep., pages
  127--151. Eur. Math. Soc., Z\"urich, 2011.

\bibitem{Stein}
Elias~M. Stein and Rami Shakarchi.
\newblock {\em Complex analysis}, volume~2 of {\em Princeton Lectures in
  Analysis}.
\newblock Princeton University Press, Princeton, NJ, 2003.

\bibitem{Lower}
Eliran Subag.
\newblock A lower bound for the mixing time of the random-to-random insertions
  shuffle.
\newblock {\em Electron. J. Probab.}, 18:no. 20, 20, 2013.

\bibitem{JC}
J.-C. Uyemura-Reyes.
\newblock Random walks, semi-direct products and card shuffling.
\newblock {\em ProQuest LLC, Ann Arbor, MI, 2002, Thesis (Ph.D.)-Stanford
  University. MR-2703300}.

\end{thebibliography}

\end{document}